\documentclass[12pt]{amsart}
\usepackage{amsfonts,amssymb,eucal}

\usepackage{amsthm}
 \usepackage{amsmath}
 \usepackage{latexsym}
\usepackage{bbold,mathbbol,bbm}
\usepackage{color}

\definecolor{orange}{RGB}{255,127,0}

\definecolor{blue}{RGB}{17,84,156}

\numberwithin{equation}{section}
\newcommand{\ep}{\ensuremath{\epsilon\,}}

\newcommand{\e}{{\mbox{\rm e}}}

\newcommand{\mb}[1]{{\mbox{\boldmath{$#1$}}}}
\newcommand{\mc}[1]{{\mathcal{#1}}}
\newcommand{\got}[1]{{\mathfrak{#1}}}
\newcommand{\db}[1]{{\mathbb{#1}}}
\newcommand{\pa}{\partial}

\newcommand{\R}{\ensuremath{\mathbb{R}}}\newcommand{\C}{\ensuremath{\mathbb{C}}}
 \newcommand{\N}{\ensuremath{\mathbb{N}}}

\newtheorem{Remark}{Remark}

\newtheorem{Proposition}{Proposition}
\newtheorem{lemma}{Lemma}

\newtheorem{Comment}{Comment}
\theoremstyle{definition}
\def\ii{\operatorname{i}}
\newcommand{\Ka}{K\"ahler}
\newcommand{\mr}[1]{{\mathrm{#1}}}
\newcommand{\dd}{\operatorname{d}}

\newcommand{\un}{\ensuremath{\mathbb{1}_n}}
\newcommand{\on}{\ensuremath{\db{0}_n}}

\newcommand{\h}{\ensuremath{\got{h}}}
\renewcommand{\Re}{\operatorname{Re}}
\renewcommand{\Im}{\operatorname{Im}}
\allowdisplaybreaks

\begin{document}
\title[Connection matrices and Berry phases]{Connection matrices and Berry phases \\on the homogeneous spaces $\mc{D}^J_1$ and $\mc{X}^J_1$ \\
 attached to Jacobi groups}
\enlargethispage{1cm}
\author{Elena Mirela Babalic, Stefan  Berceanu}

\address{''Horia Hulubei'' National
 Institute for Physics and Nuclear Engineering\\
Department of Theoretical Physics\\
PO BOX MG-6, Bucharest-Magurele, Romania}
 \email{mbabalic@theory.nipne.ro, Berceanu@theory.nipne.ro}

\begin{abstract}
We analyze the relation between connection matrices and Berry phases on the Siegel-Jacobi
disk $\mc{D}^J_1$ and Siegel-Jacobi upper half-plane $\mc{X}^J_1$ and we explicitly compute the
connection matrices and the Berry connections on these homogeneous spaces in different coordinates.
\end{abstract}
\maketitle
\tableofcontents
\newpage

\section{Introduction}

The complex  Jacobi group \cite{bs,ez} of degree $n$ is defined as the
semi-direct product $G^J_n=\mr{H}_n\rtimes\text{Sp}(n,\R)_{\C}$,
where $\mr{H}_n$ denotes the $(2n+1)$-dimensional Heisenberg group \cite{sbj,nou,Y02}.
To the Jacobi group $G^J_n $ it is associated  a homogeneous manifold, called the
Siegel-Jacobi  ball  $\mc{D}^J_n$ \cite{sbj},  whose points are in
$\C^n\times\mc{D}_n$, i.e. a partially-bounded space \cite{Y08,Y10}.   $\mc{D}_n$ denotes the
Siegel (open)  ball  of degree $n$. The non-compact Hermitian symmetric space
$ \operatorname{Sp}(n, \R )_{\C}/\operatorname{U}(n)$ admits a matrix realization
as a  homogeneous bounded domain \cite{helg}:
\[
  \mc{D}_n:=\{W\in  MS (n, \C ): \un-W\bar{W}>0\}.
  \]
The real Jacobi group of degree $n$ is defined as $G^J_n(\R):={\rm Sp}(n,\R)\ltimes \mr{H}_n$,
where $\mr{H}_n=\mr{H}_n(\R)$ is the real $(2n+1)$-dimensional Heisenberg group.
${\rm Sp}(n,\R)_{\C}$~and~$G^J_n$ are isomorphic to~${\rm Sp}(n,\R)$
and~$G^J_n(\R)$ respectively as real Lie groups, see \cite[Proposition~2]{nou}.

The invariant metric on the Siegel-Jacobi upper half-plane  on
$\mc{X}^J_1=\frac{G^J_1(\R)}{{\rm SO}(2)\times\R}\approx
\mc{X}_1\times\R^2$ \cite{jac1,BER77,FC,SB14} was obtained previously by
Berndt~\cite{bern84,bern} and K\"ahler~\cite{cal3,cal}.

 We determined the invariant metric on a five-dimensional homogeneous manifold
 ${\tilde{\mc{X}}^J_1}=\frac{G^J_1(\R)}{{\rm SO}(2)}\approx\mc{X}_1\times\R^3$ \cite{SB19},
 called the extended Siegel--Jacobi upper half-plane. The results of \cite{SB19}
concerning $\tilde{\mc{X}}^J_1$ have been generalized  in \cite{SB20N} to  the extended
Siegel-Jacobi upper half space
$\tilde{\mc{X}}^J_n=\frac{G^J_n(\R)}{\text{U}(n)}\approx
\mc{X}^J_n\times \R$, ~$\mc{X}^J_n\approx\C^n\times \mc{X}_n  $,
~$\mc{X}_n=\frac{\mr{Sp}(n,\R)}{\mr{U}(n)}$,~ $\N\ni n>1$.

 We recall that on homogeneous \Ka~ manifolds the Hamilton  equations  of motion
 and the Berry phase were  simultaneously  investigated
  in  \cite{sbcag,sbl,FC}, see also \cite{GO}.
  In the present paper we are interested in the same problem  of studying
  the  Berry phase on odd-dimensional manifolds, where several
 geometric structures can be introduced
 \cite{BL,BL2,boy,boga,can,Capp,MLEO,LI,sas}, see also a brief review
 in \cite[Appendix]{SB22}. In our paper \cite{SB22} we investigated
 Hamiltonian systems on manifolds with almost cosymplectic structure
 in the sense of \cite{paul}.

We recall here our interest in finding a geometric significance for the phase
 of the scalar product of coherent states \cite{perG,lis1,neeb}.
The answer to this question was given  by Pancharatnam for the Poincar\'e sphere
 \cite{Pan,swA}, see also \cite{aa}  and \cite[Proposition 5.1]{ma}
 in the language of holonomy of a loop in the projective
 Hilbert space,  and  by Perelomov \cite[page 63]{perG} for the  sphere
 $\text{S}^2=\frac{\text{SU}(2)}{\text{U}(1)}$.
    A general answer to this question using the coherent state embedding
    and the so called ''Cauchy formulas'' was given in \cite{SB2000} and \cite{SBS}.
    We studied this problem also in \cite{SB99}--\cite{SB2003}.
    Explicit calculation was presented for the compact Grassmann manifold
$G_n(\C^{n+m})=\frac{\text{SU}(n+m)}{\text{S}(\text{U}(n)\times
  \text{U}(m))}$  in \cite{SB2000}, where it was proved that
  {\it  the phase of the scalar product of two coherent states is twice
  the symplectic area of a geodesic triangle determined by the corresponding
points on the manifold and the origin of the system of coordinates},  see
 also \cite[Theorem 2.1]{clerc}. The same result is also true for the
 noncompact dual
 $\frac{\text{SU}(n,m)}{\text{S}(\text{U}(n)\times  \text{U}(m))}$ of
the compact Grassmann manifold \cite{SB99,SB99b}. In \cite{FC}  the change
of coordinates $x\rightarrow z $ in \eqref{2.1} was called
$FC$-transform  ($FC$ meaning fundamental conjecture, \cite{GV,pia,DN}). In \cite[Remark 3]{sbl} we observed that:
\[
{\text{  For~symmetric manifolds the  FC-transform   gives
  geodesics.\qquad (A)}}
\]
In \cite[Remark 1]{SB97}
we underlined that assertion (A) is true for a class of manifolds which
includes the naturally reductive spaces \cite{nomizu,atri},
\cite[page 202]{kn}. We have considered the
sequence of manifolds:
\begin{center}
  Hermitian symmetric spaces $\subset$ symmetric $\subset $ naturally
  reductive $\subset$ g. o. spaces
\end{center}
We have shown in \cite[Proposition 5.8]{SB19} that $\mc{X}^J_1$ is not
naturally reductive with respect to the balanced metric \cite{don,arr,alo}.

We recall that the Berry phase is an important object in the study of
geometric phase physics \cite{BS,swA, GO,DP}.
We have studied the Berry phase on homogeneous \Ka~ manifolds in
\cite{sbcag, sbl, FC}.

The paper is laid out as follows.   In Section  \ref{PR}  we recall the \Ka~
two-form on $\mc{D}^J_1$ and its two-parameter balanced metric image
on $\mc{X}^J_1$ is obtained by  the partial Cayley transform in Proposition \ref{PRFC}.
Section \ref{BEP} recalls our investigation of Berry phases and connection matrices
on \Ka~manifolds, and we compute the explicit formulas of connection matrices and
Berry connections on $\mc{D}^J_1$ and $\mc{X}^J_1$.

Proposition \ref{PRFC} and  Comment \ref{CM1} are improved versions
of older results, while Remark \ref{RRR} compares our approach
of the Berry phase on \Ka~ manifolds \cite{sbcag,sbl,FC}
 with the geometric phase in \cite{swA,BS,CH}.
 The new relevant  results presented in this paper are, in \S ~\ref{33},
 formulas \eqref{thetam} and \eqref{CMexp} of connection matrix $\theta_{\mc{X}^J_1}(x,y,q,p)$, formula
\eqref{ABXY} of the Berry connection on $\mc{X}^J_1$ expressed in
$(u,v)=(m+\ii n, x+\ii y)$, formulae \eqref{ABXYY}, \eqref{AA1}, \eqref{AA2} of the Berry phase in $(w,v)=(\alpha+\ii \beta, x+\ii y)$,
formula \eqref{EQ11} for the Berry connection in $(x,y,q,p)$.

This paper is a shortened and improved version of the preprint \cite{berry24}.

\

\textbf{Notation.}
We denote by $\mathbb{R}$, $\mathbb{C}$, $\mathbb{Z}$ and $\mathbb{N}$
 the field of real numbers, the field of complex numbers,
the ring of integers,   and the set of non-negative integers, respectively. We denote the imaginary unit
$\sqrt{-1}$ by~$\ii$, the real and imaginary parts of a complex
number $z\in\C$ by $\Re z$ and $\Im z$ respectively, and the complex
conjugate of $z$ by $\bar{z}$.
 We denote by ${\dd }$ the differential.
We use Einstein's summation convention, i.e.  repeated indices are
implicitly summed over.  The set of vector fields (respectively 1-forms) on real
manifolds  is denoted
by $\got{D}^1$ (respectively $\got{D}_1$). We denote a mixed tensor
contravariant of degree $r$ and covariant of degree $s$ by $\got{D}^r_s=\got{D}^r\times\got{D}_s$,
where $\got{D}^r=\underbrace{\got{D}^1\times \dots \times
  \got{D}^1}_{r}$ and $\got{D}_s=\underbrace{\got{D}_1\times \dots \times
  \got{D}_1}_{s}$ \cite[pages 13-17]{helg}. If $M$ is a complex manifold we
denote by $\got{A}^{r,s}$ the tensor fields of type $(r,s)$.
  If we denote with Roman  capital letteres the Lie  groups, then their
associated Lie algebras are denoted with the corresponding lower-case
letteres. If $\got{H}$ is a Hilbert space, than we adopt the
  convention  that the scalar product $(\cdot,\cdot)$ on
  $\got{H}\times\got{H}$ is antilinear in the first factor
$(\lambda a,b)=\bar{\lambda}(a,b),~ \lambda \in \C\setminus \{0\}
$. If $\pi$ is a representation of a Lie grup $G$
on the Hilbert $\got{H}$  and $X\in\got{g}$,  then we denote
${\bf X}:=\dd \pi (X)$  \cite{SB03,SB14,perG}. The interior  product $i_X\omega$ (interior multiplication or contraction)
of the differential form $\omega$
with $X\in\got{D}^1$ is denoted $X\lrcorner \,\omega$. We denote by
$M(n,m,\db{F})$ the set of $n\times m$ matrices with elements
in the field $\db{F}$ and  $M(n,\db{F})$ denotes $M(n,n,\db{F})$. If $X\in M(n,m,\db{F})$, then $X^t$ denotes the
transpose of $X$. We denote by $MS(n,\db{F}):=\{X\in M(n,\db{F})|X^t=X\}$. The
 conjugate transpose (or hermitian transpose)  of $ A\in
 M(q,\C)$  is  $A^H:=\bar{A}^t$,  also denoted $A^*$, $A^{\dagger}$, $A^+$.
 If $f$ is a function on $\C^n$, we write for the
total differential of $f$  $\dd f= \pa f+\bar{\pa} f$,
$\pa f=\sum_1^n{\pa_{\alpha}f}\dd z_{\alpha}$,
  where $\pa_{\alpha}f=\frac{\pa f}{\pa z_{\alpha}} $
  \cite[page 6]{GH}. If $f$ is a complex function, then by $f-{cc}$ we mean $f-\bar{f}$.

  \section{Preparation}\label{PR}

We adopt  the notation in  \cite{bs,ez} for the real  Jacobi group   $G^J_1(\R)$,
realized as  submatrices of $\text{Sp}(2,\R)$ of the form
\begin{equation}\label{SP2R}
g=\left(\begin{array}{cccc} a& 0&b &   q\\
\lambda &1& \mu & \kappa\\c & 0& d &  -p\\
         0& 0& 0& 1\end{array}\right),~ M=
    \left(\begin{array}{cc}a&b\\c&d\end{array}\right),~ \det M
   =1, \end{equation}
where
\begin{equation}\label{DEFY}
X:=(\lambda,\mu)~,~Y:=(p,q)=XM^{-1}=(\lambda,\mu) \left(\begin{array}{cc}a&b\\c&d\end{array}\right)^{-1}=(\lambda d-\mu
  c,-\lambda b+\mu a)
  \end{equation}
are related to the Heisenberg group $\rm{H}_1$ described by
$(\lambda,\mu,\kappa)$.  For  coordinatization of  the  real Jacobi
group  we adopt  the so called  $S$-coordinates
$(x,y,\theta,p,q,\kappa)$  \cite{bs}.

Simultaneously with
the Jacobi group $G^J_1(\R)$ consisting of elements $(M,X,\kappa)$, we
considered the restricted real Jacobi  group $G^J(\R)_0$  of elements $(M,X)$ \cite{jac1,SB19}.

  The action  $G^J(\R)_0\times \mc{X}^J_1\rightarrow
  \mc{X}^J_1$    (respectively  $G^J_1(\R)\times \tilde{\mc{X}}^J_1\rightarrow
  \tilde{\mc{X}}^J_1$) in Lemma \ref{LEMN} below  is extracted from \cite[Lemma 5.1]{SB19} and \cite[Lemma  1]{SB20}.

\begin{lemma}\label{LEMN}
Let $v,u\in\C$ be defined as:
\begin{equation}\label{TAUZ}
v:=x+\ii y,\quad \quad
  u:=pv+q=m+\ii n,\quad\quad ~x,y,p,q,m,n\in\R.
\end{equation}
(\Ka~calls $\tilde{\mc{X}}^J_1$ {\it{Phasenraum der Materie}}, where $v$ is
{\it{Pneuma}} and $u$ is {\it {Soma}} \cite{cal3}.)

\noindent Let $\mc{X}^J_1\!\approx\! \mc{X}_1\times\R^2$ be the Siegel--Jacobi upper half-plane,
where $\mc{X}_1\!=\!\{v \!\in\!\C| y:=\!\Im v\!>\!0\}$ is the Siegel upper half-plane,  and
$\tilde{\mc{X}}^J_1\approx\mc{X}^J_1\times\R$ denotes  the extended Siegel--Jacobi upper half-plane.
Then:

a) The action  $G^J(\R)_0\times \mc{X}^J_1\rightarrow
  \mc{X}^J_1$ is defined by:
\begin{equation}\label{AC1}
(M,X)\times
(v',u')=(v_1,u_1),\emph{\text{~where~}}v_1=\frac{a v'+b}{c
  v'+d},~u_1=\frac{u'+\lambda u'+\mu}{c v'+d}.
  \end{equation}

b) If  $u'=p'v'+q'$,~ $v'=x'+\ii y'$ as in \eqref{TAUZ}, then the
action
\begin{equation}\label{AC11}
(M,X)\times (x',y',p',q')=(x_1,y_1,p_1,q_1)
\end{equation}
is given by the formula:
\begin{equation} x_1+\ii y_1=\frac{(ax'+b)(cx'+d)+ac y'^2+\ii
  y'}{(cx'+d)^2+(cy')^2}\label{ALIGNN1},
                                  \end{equation}
and
\begin{equation}\label{AC12}
(p_1,q_1)=(p,q)+(p',q')
\left(\begin{array}{cc}a & b\\c & d\end{array}\right)^{-1}=(p+dp'-cq',q-bp'+aq').
\end{equation}

c)  The action $G^J_1(\R)\times \tilde{\mc{X}}^J_1\rightarrow
  \tilde{\mc{X}}^J_1$ is given by:
\begin{equation}\label{AC2}
\begin{split}
& (M,X,\kappa)\times
 (v',u',\kappa')  =(v_1,u_1,\kappa_1),\\
& (M,X,\kappa)\times (x',y',p',q',\kappa')  =(x_1,y_1,p_1,q_1,\kappa_1),\\
 & \kappa_1  =\kappa +\kappa' +\lambda q'-\mu p',~
(p',q')= (\frac {n'}{y'},m'-\frac{x'}{y'}n'),~
(\lambda,\mu)=(p,q)M,
\end{split}
\end{equation}
together with \eqref{ALIGNN1} and \eqref{AC12}.
\end{lemma}

Proposition \ref{PRFC} is an improved version of
\cite[(4.38), (5.8)]{SB14}, \cite[(28),~ (29)]{GAB}, \cite[Proposition 2.1]{SB19},
 \cite[Proposition 2]{SB21}, \cite[Proposition 2]{SB22}, \cite[(18),(19)]{BER7}.

Below $(w,z)\in  ( \mc{D}_1,\C)$,
$(v,u)\in (\mc{X}_1,\C)$, and
the parameters $k$ and $\nu$ come from representation theory of the
Jacobi group: $k$ parametrizes the positive discrete series of ${\rm
  SU}(1,1)$, $2k\in\N$, while $\nu>0$ labels the representations of
the Heisenberg group \cite{jac1}.

\begin{Proposition} \label{PRFC}

Perelomov's coherent state vectors associated to the group $G^J_1$
are defined as:
\begin{equation}\label{csu}
e_{z,w}:=e^{\sqrt{\nu}\,z\,{\mb{a}}^{\dagger}+w{\mb{K}}_+}e_0, ~z\in\C,~ |w|<1
\end{equation}
and the reproducing kernel $K = K(\bar{z},\bar{w},z,w)$ is given by:
\begin{equation}\label{hot}
K =\!(e_{z,w},e_{z,w})\!=\!
(1\!-\!w\bar{w})^{-2k}\exp{\Big[\nu\frac{2z\bar{z}\!+\!z^2\bar{w}\!+\!\bar{z}^2w}{2(1-w\bar{w})}\Big]},\quad
z ,w\in\C,~|w|<1 .
\end{equation}
We can write \eqref{hot} in a simplified form as:
\begin{equation}\label{KPF} K\!=\!P^{-2k}\exp{[\nu F]}, ~~~ P:=\!1\!-\!w\bar{w},~~~ F:=\frac{2z\bar{z}\!+\!z^2\bar{w}\!+\!\bar{z}^2w}{2(1-w\bar{w})},~~~z ,w\in\C,~|w|<1 .
\end{equation}

  a) The \Ka~two-form on  $\mc{D}^J_1$, invariant to the action of
  $G^J_1=\rm{SU}(1,1)\ltimes\C$, is
 \begin{equation}\label{kk1}
  -\ii \omega_{\mc{D}^J_1}
(w,z)\!=\!\frac{2k}{P^2}\dd w\wedge\dd
  \bar{w}\!+\!\nu \frac{\mc{A}\wedge\bar{\mc{A}}}{P},\quad \mc{A}=\mc{A}(w,z):=\dd
  z\!+\!\bar{\eta}\dd w.
\end{equation}
We have the change of variables $(w,z)\rightarrow (w,\eta, \bar{\eta})$
\begin{gather}\label{E32}
{\rm FC}\colon \
 z=\eta-w\bar{\eta},\qquad {\rm FC}^{-1}\colon \
 \eta=\frac{z+\bar{z}w}{P},
\end{gather}
\begin{equation}\label{E32a}
{\rm FC}\colon \ \mc{A}(w,z)\rightarrow\mc{A}(w,\eta,\bar{\eta}):= \dd \eta -w\dd
\bar{\eta},
\end{equation}
and thus we get:
\begin{equation}
\label{E32b}
  -\ii \omega_{\mc{D}^J_1}(w,\eta) =
  -\ii {\rm{FC}}^*(\omega_{\mc{D}^J_1}(w,z))=\frac{2k}{P^2}\dd w\wedge\dd
                                 \bar{w}+\nu\dd\eta\wedge\dd
                                    \bar{\eta}.
 \end{equation}

Defining:
\begin{equation}\label{WVAB}
  w:=\alpha+\ii \beta, ~\quad \eta:=q+\ii p,\quad \alpha,\beta, p,q\in \R,
\end{equation}
relation \eqref{E32b} becomes:
\begin{equation}\label{E32bb}
 \omega_{\mc{D}^J_1}(\alpha,\beta,q,p)  =4k\frac{\dd \alpha
                                            \wedge\dd \beta}
   {(1-\alpha^2-\beta^2)^2}+2\nu \dd q\wedge \dd p.
 \end{equation}

With \eqref{hot}, \eqref{KPF} and \eqref{E32} we  have:
\begin{equation}\label{GGG}
  K(w,\bar {w}, \eta,\bar{\eta})=(1-w\bar{w})^{-2k}\exp\Big\{\nu \Big[\eta\bar{\eta}-\frac{\bar{w}\eta^2+w\bar{\eta}^2}{2}\Big]\Big\}.
  \end{equation}

    In \eqref{GGG} we make the change of coordinates $w\rightarrow v$, as in \eqref{210b}, and using $v\!=\!x+\ii y$ and $\eta=q+\ii p$ we have in particular:
\[
  \bar{\eta}^2w=\frac{(q^2-p^2-2\ii qp)(x^2+y^2-1-2\ii x)}{N}.
  \]
  One can easily check that we get the \Ka ~potential:
  \begin{subequations}\label{FXY}
      \begin{align}
   & f(x,y,q,p)=\log K (x,y,q,p) = -2k \log P +\nu F~,\\
   & with~~~P\!=\!\frac{4y}{N},~F\!=\!\frac{2}{N}[(y\!+\!1)q^2\!+\!(x^2\!+\!y^2\!+\!y)p^2\!+\!2qpx],~N\!=\!x^2 \!+\!(y+1)^2.
       \end{align}
    \end{subequations}

    The matrix of the balanced metric $h(z,w)$ associated to the \Ka~two-form
\eqref{kk1} through a formula of the type \eqref{KALP}:
$$ -\ii\omega_{\mc{D}^J_1}(w,z)=h_{z\bar{z}}\dd z \wedge\dd\bar{z}+h_{z\bar{w}}\dd z \wedge\dd\bar{w}-h_{\bar{z}w}\dd \bar{z} \wedge\dd w+h_{w\bar{w}}\dd w \wedge\dd\bar{w},$$
is given by:
\begin{equation}\label{metrica}
  h(z,w) =\left( \begin{array}{cc}h_{z\bar{z}}& h_{z\bar{w}}\\
                         h_{w\bar{z}}& h_{w\bar{w}}\end{array} \right)=
  \left(\begin{array}{cc} \frac{\nu}{P} &  \frac{\nu\eta}{P} \\
\frac{\nu\bar{\eta}}{P} &
\frac{2k}{P^2}+\frac{\nu|\eta|^2}{P}\end{array}\right).
\end{equation}
The inverse of the matrix \eqref{metrica} reads:
\begin{equation}\label{hinv}
h^{-1}(z,w)= \left(\begin{array}{cc}h^{z\bar{z}}&
      h^{z\bar{w}}\\h^{w\bar{z}}&h^{w\bar{w}}\end{array}\right)  =  \left(\begin{array}{cc}
    \frac{P}{\nu}+\frac{P^2|\eta|^2}{2k} & -\frac{P^2\eta}{2k} \\
-\frac{P^2\bar{\eta}}{2k} & \frac{P^2}{2k}\end{array}\right).
\end{equation}

b) The second partial Cayley transform $\Phi_1: \mc{D}^J_1\rightarrow
\mc{X}^J_1$
\begin{equation}\label{PHH1}
  \Phi_1:={\rm{FC}}_1\circ\Phi:  (w,z)\rightarrow
(v=x+\ii y,\eta=q+\ii p)\end{equation}
 and its inverse $\Phi_1 ^{-1}: (v,\eta)\rightarrow (w,z)$ are  given  by
\begin{subequations}\label{ULTRAN}
\begin{align}
 \Phi_1: & ~w=\frac{v-\ii}{v+\ii}, \quad  z=\eta-\bar{\eta}\frac{v-i}{v+\ii}=2\ii \frac{pv+q}{v+\ii},\label{ULTRAN1}\\
\Phi_1 ^{-1}: &~ v=\ii \frac{1+w}{1-w}, \quad \eta =
\frac{(1+\ii \bar{v})(z-\bar{z})+v(\bar{v}-\ii)(z+\bar{z})}{2\ii
  (\bar{v}-v)}=\frac{z+\bar{z}w}{P},\label{ULTRAN2}
\end{align}
\end{subequations}
where we have:
\begin{equation}\label{Pv}
P=1-w\bar{w}=\frac{2\ii (\bar{v}-v)}{(v+\ii)(\bar{v}-\ii)}=\frac{2\ii (\bar{v}-v)}{|v+\ii|^2}.
\end{equation}
Introducing the second partial Cayley transform \eqref{ULTRAN1} and
the change of variables \eqref{E32} into the \Ka ~two-form \eqref{kk1}
on $\mc{D}^J_1$, we get the symplectic two-form \eqref{214a} on the
Siegel-Jacobi upper half-plane described by $(v,\eta)$, where $\Im v>0$:
\begin{subequations}\label{omSYUHP}
  \begin{align}
  -\ii \omega_{\mc{X}^J_1}(v,\bar{v},\eta,\bar{\eta}
    )&\!=\!\frac{8k}{P^2}\frac{\dd v\wedge \dd \bar{v}}{N^2}\!+\!\nu \dd
       \eta\wedge\dd \bar{\eta}=\!-\frac{2k}{(\bar{v}\!-\!v)^2}\dd v\wedge
       \dd \bar{v} \!+\! \nu \dd\eta\wedge\dd \bar{\eta},\label{214a}\\
     \omega_{\mc{X}^J_1}(x,y,q,p )    & =\frac{k}{y^2}\dd x\wedge \dd
                                         y+2\nu \dd q\wedge \dd p,\label{214b}
 \end{align}
 \end{subequations}
in which we used notations:
 \begin{equation}\label{NNN}
  N:=|v+\ii|^2=x^2+(y+1)^2,\quad P=4\frac{y}{N}.
 \end{equation}

 c) Using the partial Cayley transform
 $\Phi^{-1}:\mc{D}^J_1\rightarrow \mc{X}^J_1, ~(w,z)\rightarrow (v,u)$ and its
inverse
\begin{subequations}\label{210}
\begin{align}
\Phi^{-1}: v & =\ii \frac{1+w}{1-w},~~u=\frac{z}{1-w}, ~~w,z\in\C,~
               |w|<1,\label{210a}\\
  \Phi: w & =\frac{v-\ii}{v+\ii},~~z=2\ii
        \frac{u}{v+\ii},~~v,u\in\C,~\Im v>0,\label{210b}
\end{align}
\end{subequations}
and also relation \eqref{kk1} we obtain:
\begin{equation}\label{ALEFT}
\mc{A}\left(\frac{v - \ii}{v+ \ii},\frac{2\ii
      u}{v + \ii}\right)=\frac{2\ii}{v+\ii}\mc{B}(v,u),
\end{equation}
   where:
  \begin{equation}\label{BFR2}
  \mc{B}(v,u) := {\rm d} u - r{\rm d} v, \quad
  r:=\frac{u-\bar{u}}{v-\bar{v}}.
\end{equation}
  The Berndt--\Ka~ two-form (respectively the symplectic two-form), invariant to the action of
$G^J(\R)_0$ $= \rm{SL}(2,\R)\ltimes\C$, is \eqref{BFR} \emph{(}respectively \eqref{BRF}\emph{)} is:
\begin{eqnarray}
&& - \ii \omega_{\mc{X}^J_1}(v,u)  \!=\! -\frac{2k}{(\bar{v} - v)^2} \dd
v\wedge \dd\bar{v}+ \frac{2\nu}{\ii(\bar{v} -
                                 v)}\mc{B}\wedge\bar{\mc{B}} \label{BFR}\\
     &&\quad\quad\quad\quad \!=\!\frac{1}{y}\{(\frac{k}{2y}+\nu r^2)\dd v\wedge \dd
       \bar{v}+\nu[\dd u\wedge \dd \bar{u}-r(\dd v\wedge
       \dd\bar{u}-\dd \bar{v}\wedge \dd u)]\}, \label{BFFR}\nonumber\\
&&  \omega(x,y,m,n) \!=\! \frac{k}{y^2}\dd x\wedge \dd
                      y+\frac{2\nu}{y}(\dd m-r\dd x)\wedge (\dd n-r\dd
                      y)\label{BRF} \\
&&  \quad\quad\quad\quad \! =\! (\frac{k}{y^2}\!+\!2\nu \frac{r^2}{y})\dd x\wedge \dd
     y\!+\!2\frac{\nu}{y}[\dd m\wedge\dd n\!+\!r(\dd y\wedge \dd m\!-\!\dd
     x\wedge \dd n)],\label{BRF2}\nonumber
\end{eqnarray}
where, from \eqref{TAUZ} and \eqref{BFR2}, we have:
\begin{equation}\label{umn}
  u = m+\ii n, \quad m,n \in \R, \quad r= \frac{n}{y}.
\end{equation}
From \eqref{umn} and \eqref{TAUZ} we have:
 \begin{equation}
 \label{rmn}
 r=p,\quad  m=p x+q,\quad n=py .
 \end{equation}
With \eqref{hot} and \eqref{210b} we get:
\begin{equation}\label{DOIIi}
  K(\frac{v-\ii}{v+\ii},\frac{2\ii u}{v+\ii})=\left[\frac{|v+\ii|^2}{2\ii(\bar{v}-v)}\right]^{2k}\!\!\exp\Big\{\frac{2\nu}{|v+\ii|^2}\Big[|u|^2-\frac{(u\bar{v}-\bar{u}v)^2+(\bar{u}-u)^2}{2\ii(\bar{v}-v)}\Big]\Big\}.
\end{equation}

With the change of variables
${\rm FC}_1\colon (v,u)\rightarrow (v,\eta)$
 \begin{equation}\label{FC1MIN}
 {\rm FC}_1\colon \ 2\ii u=(v+\ii)\eta-(v-\ii)\bar{\eta}, \qquad
   {\rm FC}^{-1}_1 \colon  \eta=\frac{u\bar{v}-\bar{u}v +\ii (\bar{u} - u)}{\bar{v}-v},  \end{equation}
 applied to \eqref{BFR} and with \eqref{WVAB} we get:
 \begin{equation}
 \label{Bveta}
   \mc{B}(v,\bar{v},\eta,\bar{\eta})=\frac{1}{2\ii}\big[(v+\ii )\dd
     \eta-(\bar{v}-\ii)\dd \bar{\eta}\big],
     \end{equation}
   and finally we regain \eqref{214a}.

   The matrix corresponding  to the balanced  metric \eqref{NEWMM}
    associated with the   \Ka~ two-form \eqref{BFR} reads:
     \begin{equation}\label{kmb}
       h(v,u)
       =\left(\begin{array}{cc}h_{v\bar{v}}&h_{v\bar{u}}\\\bar{h}_{v\bar{u}}&h_{u\bar{u}}\end{array}\right)
       =\left(\begin{array}{cc}   \frac{k}{2y^2}+\frac{\nu r^2}{y}
                &-\frac{\nu r}{y}\\-\frac{\nu r}{y} & \frac{\nu}{y}
     \end{array}\right), ~ ~y:=\frac{v-\bar{v}}{2\ii},
\end{equation}
while its inverse is:
\begin{equation}\label{kmbINV}
       h^{-1}(v,u)=
      \left(\begin{array}{cc}h^{v\bar{v}}&h^{v\bar{u}}\\\bar{h}^{u\bar{v}}&h^{u\bar{u}}\end{array}\right)
       =\left(\begin{array}{cc}\frac{2y^2}{k}&\frac{2y^2r}{k}\\\frac{2y^2r}{k} &
     \frac{y}{\nu}+\frac{2r^2y^2}{k}\end{array}\right).
\end{equation}

 d) If we apply to \eqref{BFR2} the change of coordinates $\mc{D}^J_1\ni
 (v,u)\rightarrow (x,y,p,q)\in\mc{X}^J_1$ given in \eqref{TAUZ} and \eqref{rmn}, then
\begin{equation}\label{BUVpq}
\mc{B}(v,u)=\dd u -p \dd v
\end{equation}
becomes:
\begin{equation}\label{BUVpq1}
 \mc{B}(x,y,p,q)= v\dd p+\dd q=(x+\ii y)\dd p  +\dd q
\end{equation}
and we regain  \eqref{214b}.

e) The two-parameter   balanced  metric  on the
Siegel--Jacobi upper half-plane $\mc{X}^J_1$  associated to the \Ka~
two-form \eqref{BFR} is\footnote{Parameter $\alpha$ used here has nothing to do with variable $\alpha$ in \eqref{WVAB}}:
\begin{subequations}\label{METRS2}
  \begin{align}
  \!\!\dd s^2_{\mc{X}^J_1}(x,y,p,q)  \!&=\!
\alpha\frac{\dd x^2\!+\!\dd   y^2}{y^2}\!+\!\frac{\gamma}{y}(S\dd p^2\!+\!\dd q^2+2x\dd
                                     p\dd q)\\
    \!&=\!\alpha\frac{\dd x^2\!+\!\dd   y^2}{y^2} \!+\!
        \frac{\gamma}{y}(A^2\!+\!\!B^2),\\
 with ~~ \alpha:=k/2, &~\gamma:=\nu,~~ S:=x^2+y^2,~ A\!=\!x\dd p\!+\!\dd q,~B\!=\!y\dd p.  \label{AK}
 \end{align}
\end{subequations}

  The matrix form of the metric associated  with \eqref{METRS2} is:
\[
g_{{\mc{X}}^J_1}\! = \!\left(\begin{array}{cccc}g_{xx} &0 &0 &0\\
0& g_{yy}& 0& 0 \\
0& 0& g_{pp} & g_{pq} \\0 & 0& g_{qp}& g_{qq}
 \end{array}\right),\!
 \begin{array}{ccc}\quad g_{xx}\!=\frac{\alpha}{y^2}, &
g_{yy}\!=\!\frac{\alpha}{y^2}, &\\
\quad g_{pq}\!=\!\gamma\frac{x}{y} , &
g_{pp} \!=\!\gamma\frac{S}{y},& g_{qq}\!=\!\frac{\gamma}{y}.
\end{array}
\]
\end{Proposition}

In the following we reproduce Comment 5.5 from \cite{SB19} with some
completions:
\begin{Comment}\label{CM1}
Berndt  \cite[p. 8]{bern84}  considered the closed two-form
$\omega=\dd \bar{\dd} f'$
of Siegel--Jacobi upper half-plane $\mc{X}^J_1$,  $G^J(\R)_0$-invariant to the action \eqref{AC1},
 obtained from the K\"ahler potential
\begin{equation}\label{POT}
f'(\tau,z)= c_1\log (\tau-\bar{\tau}) -\ii
c_2\frac{(z-\bar{z})^2}{\tau-\bar{\tau}}, ~c_1,c_2>0.\end{equation}
Formula \eqref{POT} is  presented by Berndt as
``communicated to the author  by \Ka''. Our
equation \eqref{METRS2} corrects two
printing errors in Berndt's paper.

In \cite[\S~ 36]{cal3}, reproduced also
in \cite{cal}, \Ka~ argues how to choose the  potential as in
\eqref{POT},
see  also  \cite[(9)  \S ~ 37]{cal3},
where  $c_1=-\frac{k}{2}$, $c_2=\nu\pi$, i. e.
\begin{equation}\label{POT1}
  f'(\tau,z)= -\frac{k}{2}\log\frac{ \tau-\bar{\tau}}{2\ii}
  -\ii\pi\nu \frac{(z-\bar{z})^2}{\tau-\bar{\tau}}.\end{equation}

Once the \Ka~ potential \eqref{POT1} is known, we apply the recipe
\eqref{KALP2}
$$-\ii
\omega_{\mc{X}^J_1}(\tau,z)=f'_{\tau\bar{\tau}}\dd
\tau\wedge\dd \bar{\tau}+f'_{\tau\bar{z}}\dd \tau\wedge\dd \bar{z}
+f'_{z\bar{\tau}}\dd z \wedge \dd \bar{\tau}+f'_{z\bar{z}}\dd z\wedge
\dd \bar{z}.$$

 The metric {\emph{(8)}} in \cite{cal3}  differs from  the metric
 \eqref{METRS2} by a factor of 2,
 since  the Hermitian metric used by  \Ka~  is
 $~\dd s^2=2g_{i\bar{j}}\dd z_i\otimes \dd\bar{z}_j$.

 If in \eqref{POT1} we take $k/2\rightarrow k$, we have:
\begin{subequations}
  \begin{align*}
 f'_{\tau}\!=\!\frac{\partial f'(\tau,z)}{\partial \tau}&\!=\!-k\frac{1}{\tau-\bar{\tau}}+\ii \pi \nu
    \frac{(z-\bar{z})^2}{(\tau-\bar{\tau})^2}, \quad
    f'_{\tau\bar{\tau}}=-k\frac{1}{(\tau-\bar{\tau})^2}+2\ii
            \pi\nu\frac{(z-\bar{z})^2}{(\tau-\bar{\tau})^3}, \\
 f'_{\tau\bar{z}} =&-2\ii
    \pi\nu\frac{z-\bar{z}}{(\tau-\bar{\tau})^2}=f'_{z\bar{\tau}}, \quad
    f'_{z}  =-2\ii\pi\nu\frac{z-\bar{z}}{\tau-\bar{\tau}}, \quad f'_{z\bar{z}}=2\ii\pi\nu\frac{1}{\tau-\bar{\tau}},
 \end{align*}
\end{subequations}
and  we finally get  \eqref{BFR}.
 Relation \eqref{BFR}
 has been  obtained by Berndt  \cite[p. 30]{bern}, where the
 denominator of the first term is misprinted  as $v-\bar{v}$
(or $\tau-\bar{\tau}$ in the above notations). Equation
 \eqref{METRS2} appears also in  \cite[p. 30]{bern} and \cite[p. 62]{bs}.

In \eqref{POT1} we denote $(\tau,z)$  with $(v,u)$ as in \cite[(9.16)]{jac1}.
 We successively make the partial Cayley
 transform \eqref{210}, i.e. $\Phi:~(w, z) \rightarrow
 (v,u)$, a holomorphic transform, and we get \eqref{222a}.
 Then we apply  the $FC_1$ transform  \eqref{FC1MIN}, s.t. $(v,u) \rightarrow  (v,\eta)$,  a non-holomorphic transform, to obtain \eqref{222b},
 and finally we  make the symplectic transform
 $(w,z)\rightarrow (x,y,q,p)$, obtaining the result \eqref{222c}:
 \begin{subequations}\label{Ktau2}
\begin{align}
 & f''(v,u)= \log K(v,u)  =-\frac{k}{2}\log
   \frac{v-\bar{v}}{2\ii}-\ii \nu
                              \frac{(u-\bar{u})^2}{v-\bar{v}},
                 \label{222a}\\
& f''(v,\eta)= \log K(v,\eta)  =
              -\frac{k}{2}\log
   \frac{v-\bar{v}}{2\ii}
              +\frac{\ii \nu}{4} (\eta-\bar{\eta})^2(v-\bar{v})\label{222b},\\
 & f''(x,y,q,p)=\log K(x,y,q,p)  =- \frac{k}{2}\log y
  +2\nu yp^2.\label{222c}
\end{align}
\end{subequations}
Note that \eqref{222a} is different from \eqref{DOIIi}.

The metric  associated to the \Ka~two-form \eqref{214b} is:
\begin{equation}\label{newM}
 \dd s^2 (x,y,q,p)=\frac{k}{2y^2}(\dd x^2 +\dd y^2) +\nu (\dd q^2
  +\dd p^2).
\end{equation}

The metric corresponding to the \Ka~two-form \eqref{222a} is:
\begin{subequations}\label{NEWMM}
  \begin{align}
  \dd  s^2(x,y,n,m) & \!=\!\big(\frac{k}{2}\!+\!\nu \frac{n^2}{y}\big)\frac{\dd x^2\!+\!\dd y^2}{y^2}
                      \!+\!\frac{\nu}{y}\big[\dd n^2 \!+\!\dd m^2 \!-\!2\frac{n}{y}\big(\dd m \dd x\!+\! \dd n \dd  y\big)\big]\\
                    &=\frac{k}{2}\frac{\dd x^2+\dd y^2}{y^2}+\frac{\nu}{y}\big[
                      \big(\frac{n}{y}\dd x -\dd m\big)^2+\big(\frac{n}{y}\dd
                      y-\dd n\big)^2\big].
    \end{align}
     \end{subequations}

The  metric \eqref{NEWMM} corresponds to the \Ka~potential $f''(v,u)$ given in \eqref{222a} instead of $f'(v,u)$ in \eqref{POT1}.

Equation \eqref{222c} was given in \cite[(9.20)]{jac1}.
 In \cite[(4.3)]{gem}, see also \cite[Proposition
 4.1]{gem},  we presented a generalization of \eqref{222c} for $\mc{X}^J_n$, obtained by Takase in \cite[\S 9]{tak}.

 Yang calculated  in \cite{Y07}  the
  metric on $\mc{X}^J_n$, invariant to the action of $G^J_n(\R)_0$.  The equivalence of the metric of Yang with the metric
  obtained via CS  on $\mc{D}^J_n$ and then transported  to
  $\mc{X}^J_n$ via the partial Cayley transform $(v,u)\rightarrow
  (v,\eta)$ given in \eqref{210} is underlined in
  \cite{nou}. In particular, the metric {\emph{(5.21b)}}
  in  reference \cite{SB19}
 appears in \cite[p. 99]{Y07} for the particular values $c_1=1$,
 $c_2=4$. See also \cite{Yan,Y08,Y10}.
\end{Comment}

\begin{Remark}\label{REM1}
 In formula \eqref{222a} we make  the change of coordinates $FC_1$
  \eqref{FC1MIN} and we get \eqref{222b}, to which we apply
  the recipe in \eqref{KALP2} or \cite[(7.18)]{jac1} to calculate:
\[
  -\ii \omega(v,\eta)=h_{v\bar{v}}\dd v\wedge
  \dd \bar{v}+h_{v\bar{\eta}}\dd v\wedge \dd
  \bar{\eta}+h_{\eta\bar{v}} \dd\eta \wedge \dd \bar{v}
  +h_{\eta\bar{\eta}}\dd \eta\wedge \dd\bar{\eta}.
\]
The associated matrix
\begin{equation}\label{hs}
  h=\left(\begin{array}{cc}h_{v\bar{v}} & h_{v\bar{\eta}}\\
            h_{\eta
            \bar{v}}& h_{\eta\bar{\eta}}\end{array}\right)=
        \left(\begin{array}{cc}\frac{k}{8y^2}&\nu p\\
                \nu p & \nu  y \end{array}\right)
       \end{equation}
is hermitian and we have
\begin{equation}\label{omM}
  \omega_{\mc{X}^J_1}(x,y,q,p)=\frac{k}{4}\frac{\dd x\wedge\dd
    y}{y^2}+2\nu[p(\dd x\wedge \dd p+\dd q\wedge \dd y)+y\dd q\wedge  \dd p],
\end{equation}
which is different from \eqref{214b} obtained introducing
\eqref{ULTRAN1} into \eqref{kk1}.
\end{Remark}

\begin{proof}
 From \eqref{222b} we find:
  \begin{subequations}
    \begin{align}
    &  h_v=\frac{\partial f(v,\eta)}{\partial v}
      =-\frac{k}{2}\frac{1}{v-\bar{v}}+\frac{\ii\nu}{4}(\eta-\bar{\eta})^2,\\
     &  h_{v\bar{v}}\!=\!\frac{\partial^2 f(v,\eta)}{\partial\bar{v}\partial v}  =
                      -\frac{k}{2}(v-\bar{v})^{-2}=\frac{k}{8}\frac{1}{y^2},
                      \\
    &  h_{v\bar{\eta}}\!=\!\frac{\partial^2 f(v,\eta)}{\partial\bar{\eta}\partial v}  =-\ii \frac{\nu}{2}(\eta-\bar{\eta})=\nu p,\\
    &  h_{\eta}=\frac{\partial f(v,\eta)}{\partial \eta}  =\frac{\ii \nu}{2} (\eta-\bar{\eta})(v-\bar{v}),\\
    &   h_{\eta\bar{v}}=\bar{ h}_{v\bar{\eta}}  =-\ii \frac{\nu}{2}(\eta-\bar{\eta})=\nu p,\\
    &  h_{\eta\bar{\eta}} \!=\!\frac{\partial^2 f(v,\eta)}{\partial\bar{\eta}\partial \eta} =\nu y
             \end{align}
    \end{subequations}
  and we thus get \eqref{hs} which is Hermitian. The conditions \eqref{EQK}  that the metric associated to \eqref{omM} be \Ka~ are met.
   \end{proof}

\begin{Remark}
 With a formula similar to \cite[(7.18)]{jac1} applied to
  \eqref{GGG},  we obtain:
    \begin{equation}\label{F1}
    -\ii \omega(w,\eta) =\frac{2k}{(1-w\bar{w})^2}\dd w\wedge\dd
                         \bar{w}+\nu [\dd\eta\wedge \dd \bar{\eta}-\bar{\eta}\dd w\wedge
                         \dd \bar{\eta} +\eta \dd \bar{w}\wedge\dd \eta],
    \end{equation}
and if we use \eqref{WVAB} we find:
\begin{eqnarray}\label{F2}
  \omega(\alpha,\beta,q,p) &=&4k\frac{\dd \alpha
                                            \wedge\dd \beta}
   {(1-\alpha^2-\beta^2)^2}+2\nu \dd q\wedge \dd p \\
   & +&2\nu [\dd q \wedge(p\dd \alpha-q\dd \beta)+\dd p\wedge(p\dd\beta+q\dd \alpha) ]\nonumber.
   \end{eqnarray}
We note that equation \eqref{F1} is different from \eqref{E32b} and \eqref{F2} is different from \eqref{E32bb}.

If we make the coordinate change $w\rightarrow
    v$ of \eqref{ULTRAN1}, the \Ka~ two-forms \eqref{F1} and \eqref{F2} become:
    \begin{equation}
      -\ii \omega(v,\eta)\!=\!\frac{k}{2y^2}\dd v\wedge \dd\bar{v}
      +\nu \Big\{\dd \eta\wedge \dd\bar{\eta}-\! 2\ii \Big[ \frac{\bar{\eta}}{(v+\ii)^2}\dd v\wedge \dd\bar{\eta}+
      \frac{\eta}{(\bar{v}-\ii)^2} \dd \bar{v}\wedge \dd \eta\Big]\! \Big\},
    \end{equation}
    \begin{eqnarray}
    \label{altaOM}
     \quad \quad \omega_{\mc{X}^J_1}(x,y,q,p)& \!\!=\!\!&\frac{k}{y^2}\dd x\wedge \dd y  +2\nu \dd q\wedge \dd p \\
      &&+\frac{4\nu}{N^2}\Big\{\big[q(x^2\!-\!(y+1)^2)\!-\!2px(y+1)\big](\dd
      x\wedge \dd q+\dd y\wedge \dd p)\nonumber\\
     && +\big[2qx(y+1)+p(x^2-(y+1)^2)\big](-\dd x \wedge \dd p +\dd y\wedge \dd q)\Big\}.\nonumber
    \end{eqnarray}
    Note that \eqref{altaOM} is different from \eqref{214b} and \eqref{omM}.
    \end{Remark}

Furthermore, we have obtained the invariant metric to the action of the Jacobi
group $G^J_1(\R)$ on the extended Siegel-Jacobi upper half-plane $\tilde{\mc{X}}^J_1$ (see \cite[Proposition 5.6,  (5.25), (5.26)]{SB19} and also
\cite[Proposition 4,  (69) ]{SB20}).

\section{Connection matrices and Berry phases on $\mc{D}^J_1$, $\mc{X}^J_1$}\label{BEP}

        \subsection{Connection matrices on $\mc{D}_1$ and $\mc{X}_1$}\label{33}

      \

\noindent a) We start with the \Ka~potential on the Siegel disk  $\mc{D}_1$:  \[ f(w)=-2k\log P
        \]
   which corresponds to the \Ka~two-form on $\mc{D}_1$ \cite[(7.21)]{jac1}
 \begin{equation}\label{jac111}
          -\ii\omega(w,\bar{w})=\frac{2k}{P^2}\dd w\wedge \dd \bar{w}.
 \end{equation}
      If we make the change of variables \eqref{210b} $w\rightarrow v$
      and apply
\begin{equation}\label{dwv}
   \dd w=2\ii \frac{\dd v }{(v+i)^2},
 \end{equation}
      we get:
      \begin{subequations}\label{OMMM}
        \begin{align}
          -\ii \omega(v,\bar{v}) & = -\frac{2k}{(v-\bar{v})^2}\dd v\wedge \dd \bar{v}= \frac{k}{2y^2}\dd v\wedge \dd \bar{v},\label{314a}\\
          \omega(x,y) &= \frac{k}{y^2}\dd x\wedge \dd y, ~~v=x+\ii y.
        \end{align}
        \end{subequations}

\

      b)   From \eqref{kk1} we get the ''matrix'' $h$ on the Siegel-Jacobi disk  $\mc{D}_1$ :
        \begin{equation}\label{hD1J}
          h_{\mc{D}^J_1}(w)=\frac{2k}{P^2}.
        \end{equation}
      Particularizing formula \eqref{CRISTU} for the Christoffel coefficients to this case where $n=1$ and $z_\alpha$ reduce to $w$, we get:
          \begin{equation}\label{WWW}
          \Gamma^w_{ww}=\frac{2\bar{w}}{P}.
        \end{equation}
  Considering the change of coordinate $w \rightarrow v$, as in \eqref{210b}, and using \eqref{CCC}, we have:
        \begin{equation}\label{dir}
         \Gamma^v_{vv}=\Gamma^w_{ww}\frac{\pa w}{\pa v}+ \frac{\pa^2
           w}{\pa v^2}\frac{\pa v}{\pa w}.
          \end{equation}
          With \eqref{dwv} we get $\frac{\pa w}{\pa v} $ and then
          \[
            \frac{\pa^2 w}{\pa v^2} =-\frac{4\ii}{(v+\ii)^3} .
          \]
          Using \eqref{dir} we get
          \begin{equation}\label{VVV}
            \Gamma^v_{vv}=\frac{2}{\bar {v}-v}=\frac{\ii}{y},
          \end{equation}
          which is correct because if we apply \eqref{CRISTU} to
          \begin{equation}
            h_{\mc{X}^J_1}(v)=-\frac{2k}{(v-\bar{v})^2}=\frac{k}{2y^2},
            \end{equation}
   we get \eqref{VVV}.
We then have:
\begin{equation}\label{VBARV}
v-\bar{v}=2\ii \frac{P}{(1-w)(1-\bar{w})}=2\ii y.
\end{equation}

            Inverse, if we consider the change of variables
            $v\rightarrow w$ and apply \eqref{CCC}:
            \[
              \Gamma^w_{ww}=\Gamma^v_{vv} \frac{\pa v}{\pa w}
              +\frac{\pa^2v}{\pa w^2} \frac{\pa w }{\pa v},
            \]
            with \eqref{VVV} we get \eqref{WWW}.

c) With definition \eqref{CM} and formula \eqref{VVV} we get
the connection matrix on $\mc{X}_1$:
\begin{equation}\theta^v_v(v)=\Gamma^v_{vv}\dd
  v=\frac{-2}{v-\bar{v}}\dd v
 \end{equation}
and, respectively, on $\mc{D}_1$ for $v \rightarrow w$ given in \eqref{ULTRAN}, using \eqref{CCM}, \eqref{VBARV} and \eqref{WWW}:
\[\theta^w_{w}(w) =\Gamma^w_{ww}\dd w= 2\frac{\bar{w}}{1-w\bar{w}}\dd w.
\]

\subsection{Connection matrices  on $\mc{D}^J_1$ and $\mc{X}^J_1$ }\label{CON}

\

We start with the Christoffel symbols in coordinates $(w,z)$, which have  the expressions \cite[(38)]{GAB}:
\begin{equation}\label{GAMM}
\begin{split}
\Gamma^z_{zz}  & =-\lambda\bar{\eta},~\Gamma^w_{zz}=\lambda, ~
\Gamma^z_{zw}=-\lambda\bar{\eta}^2+\frac{\bar{w}}{P},\\
 \Gamma^w_{wz} & =\lambda\bar{\eta},~
 \Gamma^z_{ww}=-\lambda\bar{\eta}^3, ~ \Gamma^w_{ww} =
 \lambda\bar{\eta}^2+2\frac{\bar{w}}{P},
\end{split}
\end{equation}
 where $$\lambda = \frac{\nu}{2k}.$$
The connection matrix on $\mc{D}^J_1$
is defined, as in \eqref{CM} and \cite[(40)]{GAB}, by:
\begin{subequations}
  \begin{align}\theta_{\mc{D}^J_1}(w,z):& =\left(\begin{array}{cc} \theta^w_w & \theta^z_w\\ \theta^w_z &   \theta^z_z\end{array}\right)
  =\left(\begin{array}{cc} \Gamma^w_{wz}\dd z+\Gamma^w_{ww}\dd w &
           \Gamma^z_{wz}\dd z+\Gamma^z_{ww}\dd w\\
       \Gamma^w_{zz}\dd z+\Gamma^w_{zw}\dd w& \Gamma^z_{zz}\dd z   +\Gamma^z_{zw}\dd w
   \end{array}\right)\label{33ab63}\\
~~& =\left(\begin{array}{cc} \lambda\bar{\eta}\mc{A}+2\frac{\bar{w}}{P}\dd  w
                          & -\lambda\bar{\eta}^2\mc{A} +\frac{\bar{w}}{P}\dd z\\
                          \lambda \mc{A} & -\lambda\bar{\eta}\mc{A}+\frac{\bar{w}}{P}\dd w
   \end{array}\right),
  \end{align}\end{subequations}
where $\mc{A}$ was given in \eqref{kk1}.

In variables $(u,v)\!\in\!(\C,\mc{X}_1)$, geodesic equations \eqref{geo} for metric \eqref{kmb} read  \cite[(57)]{SB21}:
\begin{equation}\label{geomic}
 \left\{
 \begin{array}{l}
 \frac{\dd^2 u}{\dd t^2}+\Gamma^u_{uu}\left(\frac{\dd u}{\dd
     t}\right)^2 +
2\Gamma^u_{uv}\frac{\dd u}{\dd t}  \frac{\dd v}{\dd t} +\Gamma
^u_{vv}\left(\frac{\dd v}{\dd t}\right)^2 =0   ,\\
  \frac{\dd^2 v}{\dd t^2}+\Gamma^v_{uu}\left(\frac{\dd u}{\dd
     t}\right)^2 +
2\Gamma^v_{uv}\frac{\dd u}{\dd t}  \frac{\dd v}{\dd t} +\Gamma
^v_{vv}\left(\frac{\dd v}{\dd t}\right)^2=0 .
     \end{array}
 \right.
\end{equation}

The  Christoffel symbols in $(u,v)$ corresponding to the
  Riemannian metric associated to the \Ka~two-form \eqref{BFR},
  extracted from \cite[(62)]{SB21} (with small corrections added), have the expressions:
\begin{equation}\label{XGAMMM}
\begin{split}
\Gamma^u_{uu}  & =\frac{\ii}{\iota}r,~\Gamma^v_{uu}=\frac{\ii}{\iota}, ~
\Gamma^u_{uv}= \frac{\ii}{2\iota}(\frac{\iota}{y}-2r^2);\\
 \Gamma^v_{vu} & =-\frac{\ii}{\iota}r,~
 \Gamma^u_{vv}=\frac{\ii}{\iota}r^3, ~ \Gamma^v_{vv} =
\frac{\ii}{\iota}(\frac{\iota}{y}+r^2),
\end{split}
\end{equation}
where:
\begin{equation}\label{CUCUV}
  \iota:=\frac{k}{\nu},\quad r=\frac{u-\bar{u}}{v-\bar{v}}.\quad y=\frac{v-\bar{v}}{2\ii}.
\end{equation}
Using \eqref{XGAMMM} in \eqref{geomic} we find the same equations as in \cite[(50)]{SB21}:
\begin{subequations}\label{ECIV_2}
\begin{align}
& \ddot{v}+\frac{\ii}{\iota}\left[\dot{u}^2-2r\dot{u}\dot{v}+
(\frac{\iota}{y}+r^2)\dot{v}^2\right]=0,\\
& \ddot{u}+\frac{\ii}{\iota}\left[
 - r\dot{u}^2+(\frac{\iota}{y}-2r^2)\dot{u}\dot{v}+r^3\dot{v}^2
 \right]=0.
\end{align}
\end{subequations}
\begin{proof}
We check below only the value for $\Gamma^u_{uv}$, since the other ones prove similarly.

We have the following metric compatibility equations:
    \[
  \left\{
    \begin{array}{cc} h_{v\bar{v}}\Gamma^v_{vu}+h_{u\bar{v}}\Gamma^u_{vu} & =\frac{\pa h_{u\bar{v}}}{\pa v}\\
          h_{v\bar{u}}\Gamma^v_{vu}+h_{u\bar{u}}\Gamma^u_{vu}&
 =\frac{\pa h_{u\bar{u}}}{\pa v}
\end{array} \right. \!\!,\]
which, according to \eqref{kmb}, write as:
\[
  \left\{
    \begin{array}{cc} (\frac{k}{2y^2}+\frac{\nu
      r^2}{y})\Gamma^v_{vu}-\frac{\nu r}{y}\Gamma^u_{vu} & =-\frac{\ii    \nu r}{y^2}\\
             -\frac{\nu r }{y}\Gamma^u_{vu}+\frac{\nu}{y}\Gamma^u_{vu}&
 =\frac{\ii \nu}{2y^2}
    \end{array} \right. \!\!,\]
Using the Jacobian method and the following notation
\[
    \Delta=-\det h(v,u)= -\left\|\begin{array}{cc}h_{v\bar{v}}& h_{v\bar{u}}\\
               h_{u\bar{v}}& h_{u\bar{u}}\end{array} \right\|
 =-\frac{\nu k}{2y^3}~,
 \]
     \[
     \Delta_1= \left\| \begin{array}{cc}\frac{\ii\nu}{2y^2}& \frac{\nu }{y}\\
     -\frac{\ii\nu r}{y^2} & -\frac{\nu r}{y}\end{array}\right\|
     =\frac{\ii \nu ^2r}{2y^3} ~,
     \]
   \[
   \Delta_2=-\left\|\begin{array}{cc}\frac{\ii \nu}{2y^2}
                       &-\frac{\nu r}{y}\\-\frac{\ii \nu
                       r}{y^2}&\frac{k}{2y^2}+\frac{\nu r^2}{y}
                     \end{array}\right\|
                   =-\ii \frac{\nu k}{4y^4} +\frac{\ii\nu^2r^2}{2y^3}~,
       \]
                 we get
              \[
      \Gamma^v_{vu}=\frac{\Delta_1}{\Delta}=-\frac{\nu r}{k}, \quad
      \Gamma^u_{vu}=\frac{\Delta_2}{\Delta}=\frac{\ii}{2y}-\frac{\ii \nu r^2}{k}
      .    \]

    With \eqref{CRISTU} and \eqref{kmbINV}, we have:
    \[
      \Gamma^u_{vu}=h^{v\bar{u}}\frac{\pa h_{u\bar{v}}}{\pa
        v}+h^{u\bar{u}}\frac{\pa h_{u\bar{u}}}{\pa v}=
      \frac{2y^2r}{k}\frac{\pa }{\pa v}(-\frac{\nu
        r}{y})+\big(\frac{y}{\nu}+\frac{2 r^2y^2}{k}\big)\frac{\pa}{\pa v}(\frac{\nu}{y})=\frac{\ii}{2\iota}(\frac{\iota}{y}-2r^2),
      \]
which proves the value of $\Gamma^u_{vu}$ given in \eqref{XGAMMM}.
\end{proof}
The connection matrix \eqref{CM} on $\mc{X}^J_1$ in $(v,u)$ has the expression:
  \begin{subequations}\label{338}
    \begin{align}
    \theta_{\mc{X}^J_1}(v,u) &=\left(\begin{array}{cc}\theta^v_v& \theta^u_v\\
              \theta^v_u& \theta^u_u\end{array}\right)
              =\left(  \begin{array}{cc} \Gamma^v_{vu}\dd u+\Gamma^v_{vv}\dd v &   \Gamma^u_{vu}\dd u+\Gamma^u_{vv}\dd v \\
              \Gamma^v_{uu}\dd  u+  \Gamma^v_{uv}\dd  v
              &  \Gamma^u_{uu}\dd u+\Gamma^u_{uv}\dd v \end{array}\right) \label{338a} \\
     ~~ &\! =\!\frac{\ii}{\iota}   \left(\begin{array}{cc}
           -r\mc{B}+\frac{\iota}{y}\dd v   & -r^2\mc{B}+\frac{\iota}{2y} \dd u\\
   \mc{B}  & r\mc{B}+\frac{\iota}{2y}\dd v    \end{array}\right)\! \\
   &=\!\frac{\ii}{\iota}\!
                     \left[\left(\begin{array}{cc}
                           -r &-r^2 \\
                           1 & r
                           \end{array}\right)\mc{B}\!+\!
  \frac{\iota}{2y}\left(\begin{array}{cc} 2 \dd v& \dd u \\
                                          0 & \dd v\end{array}\right)\!\right]. \label{338b}
       \end{align}
  \end{subequations}
 where $\mc{B}$ was defined in \eqref{BFR2}.

 Now we calculate the connection matrix on the Siegel-Jacobi disk in
  the variables $(x,y,q,p)$. The non-zero  Christoffel symbols corresponding to the Riemannian metric
$\dd\!s^2_{\mc{X}^J_1}(x,y,q,p)$ on  the Siegel-Jacobi upper half-plane,  given in \eqref{METRS2}, are \cite[(73)]{SB21}:
\begin{equation}
  \begin{array}{lllll}\label{GSC}
    \Gamma^{x}_{xy}=-\frac{1}{y},
  &\Gamma^{x}_{pp}=-\ep xy,
  &\Gamma^{x}_{pq}= -\frac{1}{2}\ep y, & & \\
   \Gamma^y_{xx}= \frac{1}{y},&
 \Gamma^y_{yy}=-\frac{1}{y},&\Gamma^y_{pp}
 =\frac{\ep}{2}(x^2\!\!-\!\!y^2),&\Gamma^y_{pq}=\frac{\ep}{2}x,
  &\Gamma^y_{qq}=\frac{\ep}{2}, \\
  \Gamma^{p}_{xp}=\frac{1}{2}\frac{x}{y^2}, &  \Gamma^{p}_{xq}=
 \frac{1}{2}\frac{1}{y^2},&\Gamma^{p}_{yp}= \frac{1}{2y}, & & \\
  \Gamma^q_{xp}=\frac{y^2-x^2}{2y^2}, &\Gamma^q_{xq}=
  -\frac{x}{2y^2},&\Gamma^q_{yp}=-\frac{x}{y},&\Gamma^q_{yq}
  = -\frac{1}{2y} ,&  \end{array}
\end{equation}
where
$\epsilon=\frac{\gamma}{\alpha}=2\frac{\nu}{k}.$

The connection matrix in $(x,y,q,p)$ reads generally:
\begin{equation}\label{thetam}
    \!\theta_{\mc{X}^J_1} (x,y,q,p)\!=
    \!\left(\begin{array}{cccc}\theta^ x_x &\theta^ x_y &\theta^ x_q &\theta^ x_p \\
   \theta^y_x & \theta^y_y &\theta^y_q& \theta^y_p\\
  \theta^q_ x&         \theta^q_y&  \theta^q_q  &\theta^q_p\\
  \theta^p_x &  \theta^p_ y&  \theta^p_ q&  \theta^p_p
          \end{array}\right)\! .\end{equation}
With \eqref{CM} and \eqref{GSC} we find the explicit matrix elements of \eqref{thetam}:
\begin{subequations}
  \begin{align}\label{CMexp}
 &\theta_{\mc{X}^J_1} \!\! =\!\!\left(\!\!\begin{array}{cccc}\!-\frac{\dd y}{y} &
 \!-\frac{\dd x}{y}& \!-\frac{\ep}{2}y\dd p& \!-\ep xy \dd p\!-\!\frac{\ep y}{2}\dd q \\
    \frac{\dd x}{y}&\!-\frac{\dd y}{y} &
    \frac{\ep}{2}\dd q \!+\!\frac{\ep}{2}x\dd p &
    \!\frac{\ep}{2}x\dd q \!+\!\frac{\ep}{2}(x^2\!-\!y^2)\dd p\\
  \!-\frac{x}{2y^2}\dd q\!+\frac{y^2\!-\!x^2}{2y^2}\dd p
           &\!-\frac{x}{y}\dd p  &\! -\frac{x}{2y^2}\dd x &
          \frac{y^2\!-x^2}{2y^2}\dd x\!-\frac{x}{y}\dd y\\
   \frac{x}{2y^2}\dd p\!+\frac{1}{2y^2} \dd q & \frac{1}{2y}\dd p&
         \frac{1}{2y^2} \dd x &  \frac{x}{2y^2}\dd x\!+\frac{1}{2y}\dd y                               \end{array}\!\!\right) \\
   =&\! \left(\!\!\!\!\begin{array}{cccc} (0,-\frac{1}{y},0,0)
    & (-\frac{1}{y}, \!0\!,\!0,\!0)
    &(0,\!0,\!0,\!-\frac{\ep y}{2})
    &(0\!,\!0\!,-\frac{\ep}{2}y,-\ep xy)\\
     (\frac{1}{y}, \!0\!,\!0\!,\!0)
 & (0\!,\!-\frac{1}{y}\!,\!0\!,\!0)
 &(0\!,\!0\!,\!\frac{\ep}{2}\!,\!\frac{\ep x}{2})\!
 &(0\!,0\!,\frac{\ep x}{2}\!,\!
   \frac{\ep}{2} (x^2\!-\! y^2))\\
     (0,\!0,-\frac{x}{2y^2},\!\frac{y^2\!-\!x^2}{2y^2})
   &(0\!,-\frac{x}{y},\!0\!,\!0)
   & (\!-\frac{x}{2y^2},\!0\!,\!0\!,\!0)&
            (\frac{y^2\!-\!x^2}{2y^2},-\frac{x}{y},\!0\!,\!0)\!\\
      (0,\!0,\frac{1}{2y^2},\frac{x}{2y^2})&    (0\!,\!0\!,\!0\!,\!\frac{1}{2y})
      &(\frac{1}{2y^2},\!0\!,\!0\!,\!0)\!&   (\frac{x}{2y^2},\frac{1}{2y},\!0\!,\!0\!)
      \end{array}\!\!\!\!\right)
 \!\otimes\! \left(\!\!\!\begin{array}{c}\dd x\\ \dd y \\ \dd q\\ \dd  p\end{array}\!\!\!\right)\!. \end{align}\end{subequations}

\

\subsection{Berry phases on $\mc{D}_1$, $\mc{X}_1$, $\mc{D}^J_1$ and $\mc{X}^J_1$ }

\

a) First we calculate the Berry connection on $\mc{X}_1$ from the Berry connection on
$\mc{D}_1$.

Relations \eqref{210b},  \eqref{WVAB}, \eqref{TAUZ} and \eqref{NNN}
imply the following:
\begin{equation}\label{ALFAB}
  \alpha= \frac{x^2+y^2-1}{N},\quad \beta
  =-2\frac{x}{N}.
 \end{equation}
Having
\[
  f=-2k \log P=-2k\log (1-|w|^2),
\]
with \eqref{BCON} we find the Berry connection on the Siegel disk: $\mc{D}_1$
\begin{equation}\label{ThEETA}
A_B (w,\bar{w})=\ii k\frac{\bar{w}\dd w-w\dd \bar{w}}{P}=
2k \frac{\beta\dd \alpha-\alpha \dd \beta}{P}.
\end{equation}
Now, using \eqref{ALFAB}, we find the following
\begin{subequations}\label{DALPHA}
  \begin{align}
  \dd \alpha & = 2\frac{2x(y+1)\dd x+[-x^2+(y+1)^2]\dd y}{N^2},\\ \dd
    \beta & =-2\frac{[-x^2+(y+1)^2]\dd x-2x(y+1)\dd y}{N^2}.
            \end{align}
\end{subequations}
Computing
\begin{subequations}\label{312}
  \begin{align}
  -\alpha \dd \beta +\beta \dd \alpha
    = & \frac{2}{N^3}
       \{ [(x^2+y^2-1)(-x^2+(y+1)^2)-4x^2(y+1)]\dd x \nonumber\\
      & + [-2x(x^2+y^2-1)(y+1)-2x(-x^2+(y+1)^2]\dd y\}\nonumber\\
      =&  \frac{k}{2}\frac{2}{N^2}\frac{-N(x^2-y^2+1)\dd
    x-2Nxy\dd y}{y},\nonumber
  \end{align}
\end{subequations}
we find that Berry connection \eqref{ThEETA} on $\mc{D}_1$  in $(\alpha,\beta) $ implies the following formula for the Berry connection on $\mc{X}_1$ in $(x,y)$:
\begin{equation}\label{BNFDX}
   A_B(x,y)=k\frac{(-x^2+y^2-1)\dd x -2xy\dd y}{y[x^2+(y+1)^2]}.
\end{equation}
The Berry phase is related to \eqref{BNFDX} via \eqref{HOL}.

Christofell symbols in variables $(x,y)$ on $\mc{X}_1$ are
extracted from \cite[(69)]{SB21}
\begin{subequations}\label{GM22}
  \begin{align}
    \Gamma^x_{xx} & = 0,\quad  \Gamma^x_{xy}  = -\frac{1}{y},\quad
                    \Gamma^x_{yy}  = 0,\\
     \Gamma^y_{xx} & = \frac{1}{y},\quad  \Gamma^y_{xy}  = 0,\quad  \Gamma^y_{yy}  = -\frac{1}{y}.
 \end{align}
 \end{subequations}

b) {\it The Berry phase on the Siegel-Jacobi disk} $\mc{D}^J_1$ {\it {in}} $(w,z)$, $(\alpha,
\beta, q, p)$.

The starting point is the scalar product of two CS on
$\mc{D}^J_1$ \cite[(7.13b)]{jac1}
\begin{equation}\label{SCWZ}
  f(z,w)=(e_{z,w},e_{z,w})=-2k\log P+\nu F, \quad F=\frac{2z\bar{z}+\bar{w}z^2+w\bar{z}^2}{2P}.
  \end{equation}

  With \eqref{BCON} we get
  \begin{equation}\label{AA121}
    A_B(z,w)=\frac{\ii }{2}(A(z,w)-cc),
  \end{equation}
  where \cite[\S 4.2]{FC}
  \begin{equation}
  A(z,w) =(2k\frac{\bar{w}}{P}+\frac{\nu}{2}\bar{\eta}^2)\dd w
           +\nu\bar{\eta}\dd z = k(\Gamma^w_{ww}\dd w+2\Gamma^w_{wz}\dd z).\label{331b}
      \end{equation}

                   With \eqref{33ab63} we rewrite \eqref{331b} as
              \[
                 A(z,w)=k(\theta^w_w+\Gamma^w_{wz}\dd z),
                  \]
       which is in fact the formula before (4.27) in \cite{FC}.

  With \eqref{E32} we get for $A$ in \eqref{AA121}
  \[
    A(w,\eta)=(2k\frac{\bar{w}}{P}-\frac{\nu}{2}\bar{\eta}^2)\dd
    w+\nu\bar{\eta}(\dd \eta- w\dd \bar{\eta}).
    \]
We also have the following expression for the Berry connection:
\begin{equation}
  \begin{split}
      A_B(\alpha,\beta,q,p) & =2k\frac{\beta}{1-\alpha^2-\beta^2}\dd \alpha +
      [-2k\frac{\alpha}{1-\alpha^2-\beta^2}
      +\frac{\nu}{2}(q^2-p^2)]\dd\beta\\ &-\nu[(\alpha-\beta q-1)\dd q+(\alpha q+\beta)\dd p].
    \end{split}
    \end{equation}

    c) {\it{Berry phase in} } $(u,v)$

    We use \eqref{222a} for $f(u,v)$, $(u,v)=(m+\ii n,x+\ii y)$. We
    have
    \[
 f_v=-\frac{k}{2}\frac{1}{v-\bar{v}}+\ii \nu       r^2=\ii (
              \frac{k}{4y}+\nu\frac{n^2}{y^2}),\quad
              f_u =-2\ii \nu r=-2\ii \nu  \frac{n}{y}.
              \]

              We get
              \begin{subequations}
                \begin{align}
                A(u,v)& =f_u\dd u + f_v \dd v= -2\ii \nu r \dd u+\ii
                        (\frac{k}{4y}+\nu r^2)\dd v\\
                  & = k[- 2\Gamma^u_{uu}\dd
                    u+(\Gamma^v_{vv}-\frac{3i}{4}\frac{1}{y})\dd v],
                  \end{align}\end{subequations}
where we used \eqref{XGAMMM}.
      \begin{equation}\label{ABXY}
        A_B(x,y,m,n)= \frac{\ii}{2}
        (A(u,v)-\bar{A}(u,v))= \frac{1}{y}\left[-(\frac{k}{4}+\nu\frac{n^2}{y})\dd
          x+2\nu n\dd m
        \right].        \end{equation}

   d) {\it Berry phase on} $\mc{X}^J_1$ {\it in} $(v,\eta)=(x+\ii y,q+\ii p)$

      With \eqref{kk1} for $f(v,\eta)$, we have
      \[
          f_v = -2k\frac{1}{v-\bar{v}}= \ii \frac{k}{y};\quad
          f_{\eta}= -\nu (\eta-\bar{\eta})=-2\ii \nu p,
        \]
      and we find:
      \[
        A_B(x,y,q,p)=-\frac{k}{y}\dd x+2\nu p\dd q.        \]

    We write
  \begin{equation}\label{ABXYY}
    A_B(\alpha,\beta,x,y)=\dd \phi_B(\alpha,\beta,x,y)=\dd \phi_{B I}+\dd \phi_{B II}+\dd \phi_{B III},
  \end{equation}
  where $\dd \phi_{B I}$, which  appears in \eqref{ThEETA}, was
  calculated as \eqref{BNFDX}, and
  \begin{subequations}\label{AA1}
    \begin{align}\dd \phi_{B II}& = \frac{\ii \nu}{2}(-\bar{\eta}^2\dd w +cc),\\
\dd \phi_{B III}&= \frac{\ii \nu}{2}[(\bar{\eta}+w\eta)\dd \eta -cc].
    \end{align}
  \end{subequations}
  We get:
  \begin{subequations}\label{AA2}
    \begin{align}
       \dd \phi_{B II}& =\frac{\ii \nu}{2}[-(q-\ii p)^2(\dd \alpha
                        +\ii \dd \beta)+cc]=\nu [(q^2-p^2)\dd \beta
                        -2qp\dd \alpha],\\
      \dd \phi_{B III}&= \nu\{-[(\alpha+1)q+\beta p]\dd p+[(1-\alpha
                        )p+\beta q ]\dd q
                       \}.
      \end{align}
    \end{subequations}
    With \eqref{ALFAB}, \eqref{DALPHA} we find the Berry connection:
\begin{equation}\label{EQ11}
      \begin{split}
      A_B(x,y,p,q) 
   & = \frac{1}{N}\Big[ \frac{k}{y}(y^2\!-\!x^2\!-\!1)\!+\!\frac{2\nu}{N}\big[(x^2\!-\!(y\!+\!1)^2)(q^2\!-\!p^2)\!-\!4x(y\!+\!1)pq\big]\Big]\dd x\\
        & +\frac{2}{N}\{ -kx+\frac{2\nu}{N}[x^2-(y+1)^2]pq+x(y+1)(q^2-p^2)\}\dd y \\
          & +\frac{2\nu}{N}\{ [-[x^2+y(y+1)]q+xp]\dd p +[(y+1)p-xq]\dd q\}.
        \end{split}
      \end{equation}
The Berry phase is given by applying the closed contour integral \eqref{HOL} to \eqref{EQ11}.
\section{Appendix}

 \subsection{Balanced metric on homogeneous \Ka~ manifolds}

\

The starting point in  Perelomov's approach to coherent states (CS) is  the triplet
$(G,\pi,\got{H})$, where $\pi$ is a unitary, irreducible representation
of the Lie
group $G$ on a separable complex  Hilbert space $\got{H}$  \cite{perG}.

Two types of CS-vectors belonging to  $\got{H}$ are locally defined on
$M=G/H$:  the normalized (un-normalized) CS-vector
 $\underline{e}_x$ (respectively, $e_z$) \cite[\S 6, Remark 4, (6.25)]{SB95}
 \begin{equation}\label{2.1}
\underline{e}_x=\exp(\sum_{\phi\in\Delta^+}x_{\phi}{\mb{X}}^+_{\phi}-{\bar{x}}_{\phi}{\mb{X}}^-_{\phi})e_0,
\quad e_z=\exp(\sum_{\phi\in\Delta^+}z_{\phi}{\mb{X}}^+_{\phi})e_0,
\end{equation}
where $e_0$ is the extremal weight vector of the representation $\pi$,
$\Delta^+$ is the set of positive roots
of the Lie algebra $\got{g}$, and   $X_{\phi}$, $\phi\in\Delta$, and
$X^+_{\phi}$   ($X^-_{\phi}$)
 are the positive (respectively, negative) generators.

In the standard procedure of CS,
the  $G$-invariant \Ka~ two-form  on a $2n$-dimensional homogeneous
manifold $M=G/H$ is obtained from the \Ka~ potential $f$ via the recipe
\begin{subequations}\label{KALP}
  \begin{align}-\ii\omega_M & =\pa\bar{\pa}f, ~f(z,\bar{z})=\log
  K(z,\bar{z}), ~K(z,\bar{z}):=(e_{{z}},e_{{z}}),\label{KALP1}\\
\omega_M(z,\bar{z}) & =\ii \sum_{\alpha,\beta}h_{\alpha\bar{\beta}}\dd
  z_{\alpha}\wedge \dd \bar{z}_{\beta},~
  h_{\alpha\bar{\beta}}=\frac{\pa^2 f}{\pa z_{\alpha}\pa
      \bar{z}_{\beta}},~
                      h_{\alpha\bar{\beta}}=\bar{h}_{\beta\bar{\alpha}},~\alpha,\beta=1,\dots,n,\label{KALP2}
  \end{align}
  \end{subequations}
where   $K(z,\bar{z})$
is  the scalar product of two  un-normalized Perelomov's  CS-vectors $e_{{z}}$ at
$z\in M$  \cite{sbj,SB15, perG}.

It is well known, see \cite[Theorem
4.17]{ball}, \cite[Proposition 20]{SB19}, \cite[(6), p. 156]{kn},
that the condition \begin{equation}\label{condH}\dd \omega=0\end{equation} for a Hermitian
manifold to have a \Ka~ structure is equivalent with the conditions
\begin{equation}\label{EQK}
  \frac{\pa h_{\alpha\bar{\beta}}}{\pa z_{\gamma}}= \frac{\pa
      h_{\gamma\bar{\beta}}}{\pa z_{\alpha}}, \quad\text{or}\quad \frac{\pa h_{\alpha\bar{\beta}}}{\pa z_{\gamma}}= \frac{\pa
      h_{\alpha\bar{\gamma}}}{\pa z_{\bar{\beta}}},\quad\alpha,\beta,\gamma =1,\dots,n.
 \end{equation}

   In accordance with \cite[p. 42]{ball}, \cite[p. 28]{green},
   \cite[Appendix B]{SB19}, the Riemannian metric associated with the Hermitian  metric on the manifold $M$ in local coordinates is
    \begin{equation}\label{asm}\dd
    s^2_{M}(z,\bar{z})=\sum_{\alpha,\beta}h_{\alpha\bar{\beta}}\dd
    z_{\alpha}\otimes\dd \bar{z}_{\beta}.\end{equation}
  Sometimes \cite[(7.4)]{CH67}, for a  metric like \eqref{asm},  the \Ka~ two-form is taken as follows, instead of \eqref{KALP2}:
\begin{equation}\label{KALP3}
    -\ii\omega_M=\frac{\ii}{2} \sum_{\alpha,\beta}h_{\alpha\bar{\beta}}\dd
  z_{\alpha}\wedge \dd \bar{z}_{\beta}~,~~~   h_{\alpha\bar{\beta}}=\frac{\pa^2 f}{\pa z_{\alpha}\pa
      \bar{z}_{\beta}}.
 \end{equation}

The choice of $f$ in \eqref{KALP3} corresponds to the situation where the so-called $\epsilon$-function  \cite{cahII, raw,Cah} is constant:
\begin{gather*}
\epsilon(z) := \e^{-f(z)}K_M(z,\bar{z}).
\end{gather*}
The corresponding $G$-invariant metric is called {\it balanced metric}. This denomination was firstly used in~\cite{don} for compact manifolds, then in \cite{arr} for noncompact manifolds, also in~\cite{alo} in the context of Berezin quantization on homogeneous bounded domain. We used it also in the case of the partially bounded domain $\mc{D}^J_n$ -- the Siegel--Jacobi ball~\cite{SB15}.

\subsection{Christoffel symbols, geodesic equations and connection matrices}

\

We assume these notions are well known by the reader from classical references (like \cite{ccl} or \cite{kn}), and we resume here only to giving some formulas needed in the main text.

On a homogeneous space $M$ of dimension $n$, if we consider the hermiticity condition \eqref{EQK} in \eqref{KALP2}  and the K\"ahlerian restrictions \eqref{condH},  the Christoffel symbols of the Chern connection (see \cite[\S 3.2]{ball}) are determined by (see also \cite[(12), p. 156]{kn}):
 \begin{equation}\label{CRISTU} \Gamma^{\gamma}_{\alpha\beta}=\bar{h}^{\gamma\bar{\epsilon}}\frac{\pa
  h_{\beta\bar{\epsilon}}}{\pa z^{\alpha}}=
h^{\epsilon\bar{\gamma}}\frac{\pa h_{\beta\bar{\epsilon}}}{\pa  z^{\alpha}},\quad
\text{where~ ~ ~ }
h_{\alpha\bar{\epsilon}}h^{\epsilon \bar{\beta}}=\delta_{\alpha\beta}, ~~\quad~~\alpha, \beta, \gamma,
\epsilon =1,\dots,n.
 \end{equation}
 Choose any coordinate system $(U,z^\alpha)$ of the n-dimensional affine connection space $M$. The transformation formula for the Christoffel coefficients under a change of coordinates $z^\alpha\to u^i$, where $\alpha,i=1,\dots,n$,  is given by (see \cite[(2.5), p.114]{ccl}, \cite[(2), p.27]{helg}):
 \begin{equation}      \label{CCC}
 {\Gamma}^{k}_{ij} =\Gamma^\gamma_{\alpha\beta}\frac{\pa u^k}{\pa  z^\gamma}\frac{\pa z^\alpha}{\pa u^i}\frac{\pa z^\beta}{\pa u^j}
   +\frac{\pa^2z^\alpha}{\pa u^i\pa z^j}
   \cdot \frac{\pa u^k}{\pa z^\alpha},\quad i,j,k,\alpha,\beta,\gamma = 1,\ldots,n.
        \end{equation}
Geodesic equations are given by the formulas \cite[(2.20), p.116]{ccl}:
\begin{equation}\label{geo}
\frac{\dd^2 u^i}{\dd t^2}+\Gamma^i_{jk}\frac{\dd u^j}{\dd t}\frac{\dd u^k}{\dd t}=0.
\end{equation}
The elements of the connection matrix $\theta (z^\alpha)$
are given by the formula \cite[p.114]{ccl}:
\begin{equation}\label{CM}
\theta_\alpha^\beta(z^\gamma)=\Gamma^\beta_{\alpha\gamma} \dd z^\gamma.
\end{equation}
 Under the change of coordinates $z^\alpha\to u^i$ \cite[(2.4), p.113]{ccl} we have:
\begin{equation}\label{CCM}
 {\theta}^j_i(u)=\dd \left(\frac{\pa z^\alpha}{\pa u^i}\right)
 \frac{\pa u^j}{\pa z^\alpha}+\frac{\pa z^\alpha}{\pa u^i}\frac{\pa u^j}{\pa z^\beta}\theta^\beta_\alpha(z).\\
\end{equation}

\subsection{Berry phase on homogeneous \Ka~ manifolds}

\begin{Proposition}\label{PR44}
  Let $H$ be the Hamiltonian of a quantum system
  $(\Psi, \got{H}, (,))$ on the homogeneous
manifold $M=G/H$ governed by the Schr\"odinger equation
\[
  H\Psi= \ii \dot{\Psi}.
\]
Let us introduce notation
\begin{equation}\label{PPSI}
\Psi =e^{\ii \varphi}\tilde{e}_z, \quad \varphi\in [0,2\pi).
\end{equation}
Then the phase   $\varphi$ is the sum \cite{swA}
\[
  \varphi= \varphi_D+\varphi_B\]
    between the dynamical phase $\varphi_D$ and the non-adiabatic Berry phase
    $\varphi_B$, where
    \[
      \varphi_D=-\int  {\mathcal  H}(t) \dd t,
    \]
     and ${\mathcal H}$ is the energy function attached to the Hamiltonian $H$
     \begin{equation}\label{ENN}
       {\mathcal H} =(\tilde{e}_z |H| \tilde{e}_z).\end{equation}

          The Berry phase is the closed integral of the  one-form $A_B$,
          called  Berry connection:
    \begin{equation}\label{BF}
      \varphi_B = \oint A_B,
    \end{equation}
    \begin{eqnarray}&&A_B \!=\!\frac{\ii}{2}\sum_{\alpha\in
        \Delta_{+n}}
      (\dd z_{\alpha}\pa_{\alpha}\!-\!\dd {\bar{z}_{\alpha}}\bar{\pa}_{\alpha})\log
      (e_z,e_z)
      \!=\!-\Im \theta_L,\label{BCON}\\
      && \theta_L:=\! \sum_{\alpha\in
        \Delta_{+n}}\!\!\pa_\alpha f(z,\bar{z}) \dd {z}_{\alpha}\!=\!\!\sum_{\alpha\in
        \Delta_{+n}}\!\!\pa_{\alpha}\log(e_z,e_z) \dd z_{\alpha} \!=\!
      \sum_{\alpha\in
        \Delta_{+n}}\frac{\pa_{\alpha}(e_z,e_z)}{(e_z,e_z)}\dd z_{\alpha},
      \end{eqnarray}
     and $f$ is the \Ka~ potential defined in \eqref{KALP1}.
 The Berry phase depends on the path and not on the Hamiltonian.
 Closed paths in $M$ imply line integrals over connection on the closed paths and are obtained through horizontal lift.
 If the motion is done on a closed path in M, it generates
 in the fiber in M the holonomy
\begin{equation}\label{HOL}
\phi_B= \oint A_B=\int_S \dd A_B,
\end{equation}
where $\dd A_B$ is the curvature of the fiber bundle, a realisation of
Simon's two-form  \cite{BS}:
\begin{equation}\label{DDP3}
  \begin{split}
  \dd A_B& =\frac{\ii}{2}\sum_{\alpha,\beta} (-\frac{\pa^2f}{\pa {z}_{\beta}\pa
    \bar{z}_{\alpha}}\dd z_{\beta}\wedge\dd \bar{z}_{\alpha}
  +\frac{\pa^2 f}{\pa\bar{z}_{\beta}\pa z_{\alpha}}\dd
  \bar{z}_{\beta}\wedge \dd z_{\alpha})\\ &=
  -\ii \sum_{\alpha,\beta}\frac{\pa^2\log
    (e_z,e_z)}{\pa z_{\alpha}\pa \bar{z}_{\beta}}\dd z_{\alpha}\wedge
  \dd \bar{z}_{\beta}=-\omega_M(z,\bar{z}).
\end{split}
\end{equation}

  \begin{proof}
    The Proposition is taken from \cite[(4.17)]{sbcag},
        \cite[Proposition]{sbl}, \cite[corrected Proposition 4.1]{FC}, see also \cite[(15)]{GO}.
The expresion \eqref{BF} of the Berry phase corresponds to the
        parallel transport, i.e. the vector
        \begin{equation}\label{PSIBAR}
          |\underline{Z})=e^{\ii \varphi_B}\tilde{e}_z, \quad
\tilde{e}_z:=(e_z,e_z)^{-\frac{1}{2}}e_z
        \end{equation}
in \eqref{PPSI} has the property  that
 $(\underline{Z},\dot{\underline{Z}}) =0$
\cite[p. 2365]{sbl}. Also, one uses relation:
\begin{equation}\label{SUMP}
  \dot e_z=\sum_{\alpha}\frac{\pa e_z}{\pa
    z_{\alpha}}\dot  z_{\alpha} \quad ,\quad \dd e_{z}=
  \sum_{\alpha}\frac{\pa e_z}{\pa z_{\alpha}}\dd z_{\alpha}.
\end{equation}
The proof of expression \eqref{DDP3}
is a consequence of the relation:
\[
  f= \sum_{j=1}^nf_j\dd x_j\Longrightarrow \dd f= \sum_{j=1}^n\sum_{i=1}^n
  \frac{\pa f_j}{\pa x_i}\dd x_i\wedge\dd x_j,
\]
where $f$ is a smooth  function of $x_1,\dots,x_n$.
\end{proof}

\begin{Remark}\label{RRR}
  Equation \eqref{BCON} of $A_B$  can be written with formula \eqref{SUMP}
  as
  \begin{equation}\label{ruc}
    A_B=-\Im\frac{(e_z|\pa|e_z)}{(e_z,e_z)}, ~~\pa f=\frac{\pa f}{\pa
        z_{\alpha}}\dd z_{\alpha}.
    \end{equation}
    Equation  \eqref{ruc} is exactly \cite[(16) page 10]{swA} or \cite[(2.56)]{CH}
    \begin{equation}
A^{(n)}=- \Im\frac{<n|\pa|n>}{<n|n>}, \quad <n|n>=1.
   \end{equation}
Equation \eqref{DDP3} of $\dd A_B$ can be written as
    \[
      \dd A^{(n)} =-\Im \frac{(\pa e_z|\wedge|\pa e_z)}{(e_z,e_z)}.
      \]
Equation \eqref{DDP3} of $\dd A^{(n)}$ can be written as  \cite[(13) page
  10]{swA} or  \cite[(2.62), (2.63)]{CH}
\begin{equation}
  \begin{split}
    \mc{F}^{(n)}& =\dd A^{(n)}=-\Im\frac{<\dd n|\wedge|\dd n>}{<n|n>}
      =\frac{1}{2}\frac{\mc{F}^{(n)}_{ij}}{<n|n>}\dd
x_i\wedge \dd x_j\\ &=
-\Im (\frac{<\pa_i n|\pa_j n>-<\pa _jn|\pa_i n>}{<n|n>}) \dd x_i\wedge \dd x_j, <n|n>=1.
\end{split}
\end{equation}
\end{Remark}
\end{Proposition}

\subsection*{Acknowledgements}
The authors acknowledge support from the Project PN 23 21 01 01/2023 financed by the Romanian Ministry of Education and Research.

\end{document}